\documentclass{article}

\usepackage{arxiv}

\usepackage[utf8]{inputenc} 
\usepackage[T1]{fontenc}    
\usepackage{hyperref}       
\hypersetup{hidelinks}
\usepackage{url}            
\usepackage{booktabs}       
\usepackage{amsfonts}       
\usepackage{nicefrac}       
\usepackage{microtype}      
\usepackage{amsmath}
\usepackage{amsthm}
\usepackage{amssymb}
\usepackage{cleveref}       
\usepackage{lipsum}         
\usepackage{todonotes}
\usepackage{graphicx}
\newtheorem{theorem}{Theorem}
\newtheorem{proposition}[theorem]{Proposition}
\newtheorem*{proposition*}{Proposition}
\newtheorem{lemma}[theorem]{Lemma}
\newtheorem*{lemma*}{Lemma}
\newtheorem{corollary}[theorem]{Corollary}
\usepackage[backend=bibtex, style=ieee, sorting=nty]{biblatex}
\bibliography{paper.bib} 
\usepackage{esint}
\usepackage{mathtools}
\DeclarePairedDelimiter{\norm}{\lVert}{\rVert} 

\newcommand{\carac}{\scalebox{1.4}{$\chi$}}

\newcommand{\Var}{\operatorname{Var}}
\newcommand{\Vol}{\operatorname{Vol}}

\usepackage{doi}
\usepackage{changepage}

\newcommand{\mytitle}{Equidistribution of points in the Harmonic ensemble for the Wasserstein distance}
\title{\mytitle}
\newcommand{\undertitle}{A Preprint}

\date{}

\newif\ifuniqueAffiliation
\uniqueAffiliationtrue

\renewcommand{\headeright}{ }
\renewcommand{\undertitle}{ \\[1em]
\large\bfseries Pablo García Arias
\\[1em]
{\normalfont \normalsize Departament de Matemàtiques i Informàtica, Universitat de Barcelona}}
\renewcommand{\shorttitle}{\mytitle}

\hypersetup{
pdftitle={\mytitle},
pdfauthor={Pablo García Arias},
}

\begin{document}
\maketitle

\begin{abstract}
	We study the asymptotics of the expected Wasserstein distance between the empirical measure of a Point Process and the background volume form. The main DPP studied is the harmonic ensemble, where we get the optimal rate of convergence for homogeneous manifolds of dimension $d\geq 3$, and for two-point homogeneous manifolds. We also discuss some variations of this process on the torus. Regarding other point processes, we find the optimal rate for the spherical ensemble and the zeros of Gaussian Analytic Functions.
\end{abstract}

\keywords{Determinantal Point Process \and Harmonic ensemble \and Spherical ensemble \and Gaussian Analytic Function}

\section*{Acknowledgements}
The author has been supported by the grant ``Ayudas para contratos predoctorales para la formación de doctores 2022'', by the Agencia Estatal de Investigación and by the FSE+, project PID2021-123405NB-I00; by the project PID2024-160033NB-I00; and by the Departament de Recerca i Universitats, grant 2021 SGR 00087.

I would also like to thank Joaquim Ortega-Cerdà for his insights during the writing of this paper.


\section{Introduction}

Determinantal Point Processes (abbreviated DPPs) are a type of random point processes that have gained interest due to their applications in physics, their frequent appearance in random matrix theory, and their ability to get uniformly distributed points. Unlike other ways of getting random points, like choosing them from identically independently distributed random variables, DPPs have a built-in repulsion between points. Due to this, they are used to model fermionic particle systems, where the particles have all the same charge and therefore are influenced by repulsive forces. Other notable DPPs are constructed by taking eigenvalues or singular values of matrices (under some assumptions), like the Gaussian Unitary Ensemble. 

This gives a motivation to check whether for a DPP the empirical measure $\frac{1}{N} \sum_{n=1}^N \delta_{x_n}$ approximates the background measure. When it is the case, we can also try to quantify this equidistribution. There are several ways of doing it, and perhaps the most common one is through the discrepancy and the Wasserstein distance. In this paper we focus on the latter. It is worth commenting that we will work with $W_2$, while much of the previous literature studied $W_1$.

A DPP has its joint intensities given by the determinant of a positive semi-definite kernel, being the point process fully determined by a suitable choice of the kernel. A type that is especially common is when this is the reproducing kernel of some finite dimensional vector space of functions. The harmonic ensemble is constructed this way, by taking as vector space the eigenfunctions of the Laplace-Beltrami operator with eigenvalue up to a certain threshold. The study of this ensemble has been a motivation for this paper.

The key scheme for the proofs is based on the paper \cite{borda2023riesz}, where the harmonic and spherical ensemble in $\mathbb{S}^2$ and a variation of the harmonic ensemble in $\mathbb{T}^2$ were studied. The idea is to use a ``smoothing inequality'', a Berry--Esseen type inequality for the Wasserstein distance $W_2$ (see \cite{borda2022empirical} and Section \ref{section:smoothing} for the details). We rewrite the application of this inequality for a general manifold and a general point process. The main contributions of this paper are:
\begin{itemize}
	\item We get an optimal result for finite dimensional projection kernels with constant first intensity in dimension $d\geq 3$. This applies to the jittered sampling process and the harmonic ensemble in homogeneous manifolds. We present weaker results for lower dimensions and the harmonic ensemble in general compact manifolds.
	\item Through a different approach, the optimal result is achieved for the harmonic ensemble in two-point homogeneous manifolds of dimension $d\geq 2$. Here the kernel is explicit, so we can bound precisely the variance of smooth linear statistics. 
	\item We discuss variations of the Harmonic Ensemble in the torus $\mathbb{T}^2$, proving also an optimal rate of convergence. For higher dimensions, the previous results can be directly applied.
	\item We study the asymptotics of two other point processes in dimension two: the spherical ensemble and the zeros of the spherical Gaussian Analytic Functions.
\end{itemize}

\section{Background}
Throughout this paper $\mathbb{M}$ will be a compact connected Riemannian manifold without boundary. Its real dimension will be denoted by $d$. The normalized volume form is $\Vol$. When integrating with respect to it, $d\Vol(x)$ is abbreviated to just $dx$, unless it could be confused with some other standard measure. 

The Laplacian $\Delta = - \operatorname*{div} \nabla$ has a decomposition in eigenvalues $0=\lambda_0 < \lambda_1 \leq \lambda_2 \leq \cdots$ and eigenfunctions $\left\{ \varphi_m \right\}_{m=0}^\infty$ that form an orthonormal basis of $L^2(\mathbb{M})$ (see for instance \cite[Theorem 7.2.6]{Buser_2010} for the details). As there can be repetitions, sometimes it is useful to consider the equivalent list $0 = \tilde{\lambda}_0 < \tilde{\lambda}_1 < \tilde{\lambda}_2 < \cdots$ and the vector spaces $E_L = \operatorname{span} \left\{ \varphi_m \;:\; \lambda_m = \tilde{\lambda}_L \right\}$ generated by the eigenfunctions with the same eigenvalue.

A special class of compact manifolds we will work with are the homogeneous manifolds, that is, manifolds such that for any $x,y\in \mathbb{M}$ there exists an isometry $g: \mathbb{M} \to \mathbb{M} \in C^\infty(\mathbb{M})$ with $g(x) = y$. 

A two point homogeneous manifold is a Riemannian manifold such that for any two pairs of points $x_0, x_1 \in \mathbb{M}$ and $y_0, y_1 \in \mathbb{M}$ with $d(x_0, x_1) = d(y_0, y_1)$ there exists an isometry $g$ with $g(x_i) = y_i$.

The complete list of compact connected two-point homogeneous manifolds is known \cite{Wang52}:
\begin{itemize}
    \item The sphere $\mathbb{S}^d$.
    \item The projective spaces $\mathbb{RP}^n$, $\mathbb{CP}^n$, $\mathbb{H}\mathbb{P}^n$ and $\mathbb{OP}^2$.
\end{itemize}
Integrating over projective spaces is very similar to integrating over spheres, and \cite{anderson2022riesz} realizes a parallel study of them by choosing some adequate constants. What we will require to use is the fact that for an integrable function $F: \mathbb{FP}^n \to \mathbb{R}$ of the form $F(x) = f(d(x, x_0))$ we have
\begin{align*}
	\int_{ \mathbb{FP}^n} F(x) dx = \int_0^{\frac{\pi}{2}} f(r) \left( \frac{1}{\gamma_{\mathbb{M}}} \sin^{d-1 }r \cos^{\dim \mathbb{F} - 1 } r \right) \; dr
\end{align*}  
where $\dim \mathbb{F}$ means the dimension of $\mathbb{F}$ as a real vector space and $d$ the dimension of the manifold $d = n \dim \mathbb{F}$. The constant $\gamma_\mathbb{M}$ depends on the projective space we are integrating over, so that its volume is $1$.  
\subsection{Determinantal Point Processes}

A (simple) Point Process on $\mathbb{M}$ is a random variable that outputs subsets of $\mathbb{M}$ without accumulation points. A Determinantal Point Process (or DPP) is a point process whose joint intensities $\rho_k: \mathbb{M}^k \to [0, \infty)$ are given by $\rho_k(x_1, x_2, \dots, x_k) = \det \left( K(x_i, x_j) \right)_{1\leq i,j \leq k}$ for some kernel $K: \mathbb{M} \times \mathbb{M} \to \mathbb{C}$ and a background measure that is fixed as $\operatorname{Vol}$.

Note that the existence of a DPP with kernel $K$ is not guaranteed. There are conditions to impose just to get the functions $\rho_k$ to be non-negative, which is needed to get valid probability distributions. Besides the necessary basic assumptions, Macchi--Soshnikov theorem \cite[Theorem 4.5.5]{Hough_Krishnapur_Peres_2012} answers this question.

We assume that the kernel is locally square integrable on $\mathbb{M}^2$, so we can consider the operator
\begin{align*}
	\mathcal{K}: L^2(\mathbb{M}) & \longrightarrow L^2(\mathbb{M}), \\
	f & \longmapsto  \mathcal{K}f (x) = \int K(x, y) f(y) dy.
\end{align*} 
We are particularly interested in the case when $\mathcal{K}$ is a projection operator over some finite dimensional vector space $V \subseteq L^2(\mathbb{M})$. These are called (finite dimensional) projection DPP. If $\left\{ f_n \right\}_{n=1}^N$ is an orthonormal basis of $V$, we can write the kernel as
\begin{equation*}
	K(x, y) = \sum_{n=1}^N f_n(x) \overline{f}_n(y).
\end{equation*}  

It is a well-known fact that the joint intensities of a point process satisfy that for a measurable function $f: \mathbb{M}^k \to \mathbb{C}$ and random points $x_1, \dots, x_k \in \mathbb{M}$ drawn from the DPP
\begin{align*}
	\mathbb{E}\sum_{x_1, \dots, x_k \text{  distinct}} f(x_1, \dots, x_k) = \int f(x_1, \dots, x_k) \;d \rho_k(x_1, \dots, x_k) .
\end{align*}  
This can be seen in \cite[Pages 10-11]{Hough_Krishnapur_Peres_2012}. The expression we are computing the expected value of is called a statistic of the process. A particularly interesting case concerns linear statistics, that is, the random variable $\sum_{n=1}^N f(x_n)$. 

For finite projection kernels, we can compute the variance of a linear statistic like 
\begin{align} \label{eq:var}
	\Var \left(\sum_{n=1}^N f(x_n)\right) = \frac{1}{2} \iint |f(x)-f(y)|^2 |K(x,y)|^2 \;dxdy. 
\end{align}

For a more detailed explanation of Determinantal Point Processes, we recommend \cite{Hough_Krishnapur_Peres_2012}.

The main DPP we focus our attention in this paper is the \emph{Harmonic ensemble}. 
If we consider the vector space $\operatorname{span} \left\{ E_1, E_2, \dots, E_L \right\} = \operatorname{span}\left\{ \varphi_m \;:\; \lambda_m \leq \tilde{\lambda}_L \right\}$ for $L \in \mathbb{N}$ we can consider the DPP given by the projection over it. Then the kernel is   
\begin{equation*}
	K_L(x, y) := \sum_{\lambda_m \leq \tilde{\lambda}_L} \varphi_m(x) \varphi_m(y)
\end{equation*} 
for $L\in \mathbb{N}$. The number of points is denoted by $N$, that is, $N = \# \left\{ \lambda_m \leq \tilde{\lambda}_L  \right\} = \sum_{k=0}^L \dim E_k$. A useful point-wise bound for this kernel is Hörmander result \cite{Sogge_1987}:
\begin{equation*}
	K_L(x, y) \lesssim \frac{N}{1+ N^{\frac{1}{d}} d(x, y)}. 
\end{equation*}

Another interesting DPP is the spherical ensemble. In the complex plane consider the eigenvalues $z_1, \dots, z_N$ of $A^{-1}B$ for $A, B$ matrices $N\times N$ with entries being i.i.d. standard complex Gaussians. This turns out to be a Determinantal Point Process with a fixed number of points $N$. 
We usually prefer to look it in $\mathbb{C} \cong \mathbb{S}^{2} \backslash \{p\}$. For this, consider the standard stereographic projection $f: \mathbb{S}^{2} \backslash \{p\} \to \mathbb{C}$. Then the kernel can be seen to be
\begin{align*}
	K_N(x, y) = N  \frac{\left(1+ f(x) \overline{f(y)}\right)^{N-1}}{ \left( 1+ |f(x)|^2 \right)^{\frac{N-1}{2}} \left( 1+ |f(y)|^2 \right)^{\frac{N-1}{2}} }.
\end{align*}

Generalizing this ensemble to higher dimensions in a useful way is not a trivial fact. In \cite{BELTRAN20191073} the authors showed there is a \textit{reasonable} generalization to even-dimensional spheres that has constant first intensity. Another process related to it is the projective ensemble, constructed in an analogous way but in the complex projective spaces. See \cite{beltran2017projective} for the details.

The Jittered Sampling process consists in partitioning the manifold $\mathbb{M}$ into $N$ equal area parts, and choosing at random one point in each piece. This turns out to be a DPP with finite dimensional kernel: for a partition $\left\{ A_n \right\}_{n=1}^N$ with $\Vol(A_n) = \frac{1}{N}$, its kernel has the form
\begin{align*}
	K(x, y) = N \sum_{n=1}^N \carac_{A_n}(x) \carac_{A_n}(y).
\end{align*}

\subsection{Spherical Gaussian Analytic Functions}\label{section:GAF}

In chapters 2 and 3 of \cite{Hough_Krishnapur_Peres_2012} they collect the basic theory regarding random Gaussian Analytic Functions, previously studied by Sodin, Tsirelson, and other mathematicians and physicists. Here we present a quick summary of their main properties and how they can be used to get point processes. 

A Gaussian Analytic Function $f: \Lambda \to \mathbb{C}$ is an analytic function in $\Lambda \subseteq \mathbb{C}$, chosen at random with a distribution such that for any $z_1, z_2, \dots, z_N\in \Lambda$ the random vector $(f(z_1), \dots, f(z_N))$ is $N$-dimensional complex Gaussian. 

This GAF has an associated covariance kernel $K: \Lambda \times \Lambda \to \mathbb{C}$, that is characterized by the fact that $(f(z_1), \dots, f(z_N))$ has covariance matrix $(K(z_i, z_j))_{i,j = 1, \dots, N}$. 

In our case, we focus on the following family of GAFs, as they will induce a point process in the sphere invariant under rotations. For $N \in \mathbb{Z}_+$
\begin{equation*}
	f_N(z) = \sum_{n=1}^N a_n \sqrt{ \left( N \atop n \right)} z^n 
\end{equation*}
is a GAF in $\mathbb{C}$. In fact, as this is a polynomial, it can be seen as an analytic function in $\mathbb{C}\cup \left\{ \infty \right\} \cong \mathbb{S}^2$.

The roots of $f_N$ form a simple point process in $\mathbb{C}\cup \left\{ \infty \right\}$. Call them $z_1, z_2, \dots, z_N \in \mathbb{C}$. This process is invariant in distribution under the rotations of the sphere $z \mapsto \frac{\alpha z + \beta}{- \overline{\beta} z + \overline{\alpha}}$, $|\alpha|^2 + |\beta|^2 = 1$, so it preserves the spherical measure $dm^*(z) = \frac{1}{\pi(1+|z|^2)^2} dm(z)$. This motivates lifting the points through the spherical projection $z_n \mapsto x_n \in \mathbb{S}^2$. The points from this induced point process now lie in the compact manifold $\mathbb{S}^2$ and are invariant under its isometry group. 

For the first intensity of a GAF in $\Lambda \subseteq \mathbb{C}$ with kernel $K$ there is the Edelman--Kostlan formula  
\begin{align*}
	\rho_1(z) = \frac{1}{4 \pi} \Delta \log K(z, z).
\end{align*}
For the spherical GAF the covariance kernel is $K(z, w) = (1+ z\overline{w})^N$, leading to $\rho_1(z) = \frac{N}{\pi(1+|z|^2)^2}$. When the points are lifted to the sphere we get constant first intensity $\rho_1^*(x) = N$.   

\subsection{Wasserstein distance}
The Wasserstein distance of two probabilities measures $\mu$ and $\nu$ is defined as 
\begin{align*}
	W_p(\mu, \nu) = \left( \inf \int_{\mathbb{M}\times \mathbb{M}} d(x, y)^p \;d \pi(x, y) \right)^{\frac{1}{p}},
\end{align*}  
where the infimum is taken over all probabilities $\pi$ on $\mathbb{M}\times \mathbb{M}$ with $\mu$ as one marginal and $\nu$ as the other. For $p\geq 1$ this is a distance in the space of probabilities with finite $p$-order moment.

The key property of this metric is that a sequence of probabilities $\mu_N$ converges weakly $\mu_N \rightharpoonup \mu$ if and only if $W_p(\mu_N , \mu) \to 0$. This gives a way of quantifying the speed of convergence for some sequences of probability measures. Of course, there are other interesting approaches to quantify this, like the discrepancy. A standard reference for the theory of Wasserstein distance is \cite{villani_2009}.

It is worth commenting that when we have some points $\{x_n\}_{n=1}^N$ in a $d$-dimensional manifold, there is a classical lower bound for the Wasserstein distance between its empirical measure and the volume. It requires using the fact that the Wasserstein distances are increasing: $W_1 \leq W_p$ for $p \geq 1$, and then we can use the Kantorovich duality to get
\begin{align*}
    W_1 \left( \frac{1}{N} \sum \delta_{x_n} , \operatorname{Vol} \right) & \geq \int d(y, \left\{x_n\right\} ) \;dy.
\end{align*}
Consider now around each point a ball of radius $\delta >0$. Note that the volume is $\operatorname{Vol} B(x, \delta) \approx \delta^d$. Outside these balls, the distance function can be bounded by $d(y, \left\{x_n\right\} ) \geq \delta$. This gives us the bound
\begin{align*}
    W_1 \left( \frac{1}{N} \sum \delta_{x_n} , \operatorname{Vol} \right) & \geq \int_{\{d(y, \left\{x_n\right\} \}) \geq \delta} d(y, \left\{x_n\right\} ) \;dy \geq \delta \big( 1-\operatorname{Vol} \left(\{ d(y, \left\{x_n\right\} ) \leq \delta \} \right) \big) 
    \geq \delta \left( 1 - N c\delta^d \right).
\end{align*}
Finally, take $\delta^d = \frac{1}{2cN}$ to get
\begin{align*}
    W_1 \left( \frac{1}{N} \sum \delta_{x_n} , \operatorname{Vol} \right) \gtrsim \frac{1}{N^{ \frac{1}{d} }}.
\end{align*}
\subsection{Sobolev seminorms}
For functions in $L^2(\mathbb{M})$ and $0<s<1$ the fractional Sobolev spaces are defined as
\begin{align*}
	H^s(\mathbb{M}) = \left\{ f\in L^2(\mathbb{M}) \;:\; [f]_s < \infty \right\}, && [f]_s^2 = \iint \frac{|f(x)-f(y)|^2}{d(x,y)^{d+2s}} \;dxdy.
\end{align*} 
Besides the definition, it will be later used that this seminorm is known for eigenfunctions $[\varphi_m]^2_s \approx \lambda_m^s$, see \cite{Brandolini_Choirat_Colzani_Gigante_Seri_Travaglini_2014} for the computations.
We also require the following proposition. This is an adaptation of \cite[Lemma 4.2]{imbert2019weak} done in \cite[Proposition 2.14]{levi2023linear}, where we highlight that the assumption on the injectivity radius is not necessary for complete manifolds.

\begin{proposition}[{\cite[Proposition 2.14]{levi2023linear}}]\label{prop:sobolev}
    Let $\mathbb{M}$ be a $d$-dimensional complete Riemannian manifold with diameter $d_\mathbb{M}$. Assume there is a constant $c_\mathbb{M} \leq d_\mathbb{M}/2$ such that $\Vol(B(x,r)) \gtrsim r^d$ for $r\leq c_{\mathbb{M}}$. Let $K: \mathbb{M} \times \mathbb{M} \to \mathbb{R} $ be a symmetric kernel such that
    \begin{align*}
        \int_{B(x, 2r) \backslash B(x, r)} |K(x, y)| dy \leq \Lambda r^{-2s}, && \forall x\in \mathbb{M}, \forall r\in \left( 0, \frac{d_\mathbb{M}}{2} \right),
    \end{align*}        
    for some $\Lambda > 0$ and some $s\in (0,1)$. Then there exists a constant $C = C(d, s)$ such that for any $g\in H^s(\mathbb{M})$
    \begin{align*}
        \iint_{d(x,y) < 2c_\mathbb{M}} |g(x)-g(y)|^2 K(x,y) dxdy & \leq C \Lambda [g]^2_s.
    \end{align*}  
\end{proposition}

\section{Main results}
\begin{theorem}\label{thm:projectionkernels}
	Let $\mathbb{M}$ be a compact connected Riemannian manifold without boundary. Given $\left\{ x_n \right\}_{n=1}^N$ drawn from a finite dimensional projection DPP   
	\begin{align*}
		K_N(x, y) = \sum_{m=1}^N f^N_m(x) \overline{f}^N_m (y) 
	\end{align*}  
	with $\left\{ f^N_m \right\}_{m=1}^N$ orthonormal in $L^2(\mathbb{M}) $ and $K_N(x, x) = N$ constant, then
	\begin{itemize}
		\item For dimension $d\geq 3$, $\displaystyle \mathbb{E} W_2\left( \frac{1}{N} \sum \delta_{x_n} , \Vol \right) \lesssim N^{- \frac{1}{d}}$. 
		\item For dimension $d = 2$, $\displaystyle \mathbb{E} W_2\left( \frac{1}{N} \sum \delta_{x_n} , \Vol \right) \lesssim N^{- \frac{1}{2}} \sqrt{\log N}$. 
	\end{itemize}   
\end{theorem}
It is immediate that this theorem applies to the Spherical ensemble and the Jittered process, although for the first we will present a different result that gives the optimal rate. It also works in the generalizations of the spherical ensemble discussed before. The Corollary below is a consequence of applying it to the harmonic ensemble.
\begin{corollary}\label{cor:harmonichom}
	Let $\mathbb{M}$ be a compact connected homogeneous manifold and $x_1, \dots, x_N$ be points given by the harmonic ensemble. Then 
		\begin{itemize}
		\item For dimension $d\geq 3$, $\displaystyle \mathbb{E} W_2\left( \frac{1}{N} \sum \delta_{x_n} , \Vol \right) \lesssim N^{- \frac{1}{d}}$. 
		\item For dimension $d = 2$, $\displaystyle \mathbb{E} W_2\left( \frac{1}{N} \sum \delta_{x_n} , \Vol \right) \lesssim N^{- \frac{1}{2}} \sqrt{\log N}$. 
	\end{itemize} 
\end{corollary}
Moreover, when the manifold is not necessarily homogeneous, we can still give some convergence to the background measure.
\begin{theorem}\label{thm:harmoniccompact}
	For a compact connected manifold without boundary, the harmonic ensemble satisfies
	\begin{itemize}
		\item For dimension $d\geq 3$,   $\displaystyle \mathbb{E} W_2\left( \frac{1}{N} \sum \delta_{x_n} , \Vol \right) \lesssim N^{- \frac{2}{d^2}}$. 
		\item For dimension $d = 2$,   $\displaystyle \mathbb{E} W_2\left( \frac{1}{N} \sum \delta_{x_n} , \Vol \right) \lesssim N^{- \frac{1}{2}} \sqrt{\log N}$.  
	\end{itemize}  
\end{theorem}

The next theorem follows a different approach and proves the optimal rate of convergence for some families of manifolds.
\begin{theorem}\label{thm:twohomogeneus}
	For a two point homogeneous manifold $\mathbb{M}$ with dimension $d\geq 2$, the harmonic ensemble satisfies 
	\begin{equation*}
		\mathbb{E} W_2\left( \frac{1}{N} \sum \delta_{x_n} , \Vol \right) \lesssim N^{- \frac{1}{d}}.
	\end{equation*} 
\end{theorem}

\begin{theorem}\label{thm:spherical}
	For $\left\{ x_n \right\}_{n=1}^N \subseteq  \mathbb{S}^2$ given by the spherical ensemble, 
	\begin{equation*}
		\mathbb{E} W_2\left( \frac{1}{N} \sum \delta_{x_n} , \Vol \right) \lesssim N^{- \frac{1}{2}}.
	\end{equation*} 
\end{theorem}
It is worth pausing a moment on the fact that i.i.d. points attain the optimal rate for dimensions $d \geq 3$ but for $d=2$ they instead satisfy 
\begin{equation*}
    \mathbb{E} W_2\left( \frac{1}{N} \sum \delta_{x_n} , \Vol \right) \approx N^{- \frac{1}{2}} \sqrt{\log N}.
\end{equation*}
This means that the improvement due to the repulsion of a DPP is visible when $d=2$. In section \ref{section:torus} we will check that the same happens for the harmonic ensemble when the manifold is a torus.

Finally, for the zeros of the GAFs, we also obtain a better rate than i.i.d. points. This is consistent with the fact that, although it is not determinantal, the points still have a little repulsion at short distances. 
\begin{theorem}\label{thm:zerosGAF}
	For $\left\{ x_n \right\}_{n=1}^N \subseteq  \mathbb{S}^2$ given by the zeros of the spherical GAF (as explained in section \ref{section:GAF})
	\begin{equation*}
		\mathbb{E} W_2\left( \frac{1}{N} \sum \delta_{x_n} , \Vol \right) \lesssim N^{- \frac{1}{2}}.
	\end{equation*} 
\end{theorem}
\section{Proofs}
\subsection{Smoothing inequality}\label{section:smoothing}
The main tool to obtain the results is the smoothing inequality, a Berry--Esseen type inequality for the Wasserstein distance $W_2$. This was proved in \cite[Theorem 5]{borda2022empirical}. 

For points in $\left\{ x_n \right\} \subseteq \mathbb{M}$ and any $t\geq 0$ we have the inequality
\begin{equation*}
	W_2\left( \sum_{n=1}^N \frac{1}{N} \delta_{x_n}, \Vol \right) \leq \left( d \cdot t + K(\mathbb{M}) t^{\frac{3}{2}} \right)^{\frac{1}{2}} + 2 \left( \frac{1}{N^2}\sum_{m=1}^\infty \frac{e^{- \lambda_m t}}{\lambda_m} \left| \sum_{n=1}^N \varphi_m(x_n)\right|^2\right)^{\frac{1}{2}},
\end{equation*} 
where the constant $K(\mathbb{M})$ depends only on the manifold, being actually $0$ in many common cases. In any case, we are interested in looking into the asymptotics when $t \to 0$, so the term $t^{\frac{3}{2}}$  can be disregarded at the cost of (possibly) worsening the constant of the term $t$. When the points come from some random process, one can pass to the expected value using the triangle inequality for the $L^2$ norm:
\begin{align} \label{aux3}
    \mathbb{E} W_2\left( \sum_{n=1}^N \frac{1}{N} \delta_{x_n}, \Vol \right) &  \leq \sqrt{\mathbb{E} (W_2^2) } \leq \left( d \cdot t + K(\mathbb{M}) t^{\frac{3}{2}} \right)^{\frac{1}{2}} + 2 \left( \frac{1}{N^2}\sum_{m=1}^\infty \frac{e^{- \lambda_m t}}{\lambda_m}  \mathbb{E} \left( \left| \sum_{n=1}^N \varphi_m(x_n)\right|^2 \right)\right)^{\frac{1}{2}}.
\end{align}
\textbf{Remark:} 
Note that the method requires from the point process an estimate of $\mathbb{E} \left( \left| \sum_{n=1}^N \varphi_m(x_n)\right|^2 \right)$, and from this independently obtains a bound on the Wasserstein distance. This simplification is quite useful and one of the key advantages of the method, as it reduces the problem to studying the first and second intensities, rather than the $N$-point intensity. The main results from this paper have been obtained by bounding this quantity in each case and then substituting in the smoothing inequality. 

This technique is summarized in the following lemma.
\begin{lemma}\label{lemma}
    Let $\left\{ x_n \right\}_{n=1}^N \subseteq \mathbb{M}$ be from some point process. Assume the following estimates
    \begin{align*}
        \mathbb{E} \left( \left|\sum_{n=1}^N \varphi_m (x_n)\right|^2 \right) \lesssim N^a \lambda_m^b, && b+ \frac{d}{2} > 1.
    \end{align*}
    Then 
    \begin{align*}
        \mathbb{E} W_2\left( \sum_{n=1}^N \frac{1}{N} \delta_{x_n}, \Vol \right) \lesssim \frac{1}{N^\gamma}, && \gamma = \frac{2-a}{2b+d}.
    \end{align*}

	If instead we have $b+ \frac{d}{2} = 1$, we get $ \displaystyle \mathbb{E} W_2\left( \sum_{n=1}^N \frac{1}{N} \delta_{x_n}, \Vol \right) \lesssim \frac{\sqrt{\log N}}{N^\gamma}$.
\end{lemma}
\begin{proof}
	Let's focus on the case $b+ \frac{d}{2} > 1$ first. To prove it, we use the inequality
	\begin{equation*}
        \sum_{m=1}^{\infty} e^{\lambda_m t} \lambda_m^A \lesssim t^{- \left( A + \frac{d}{2} \right)}
    \end{equation*}
    for $A=b-1$. To compute the series we define the following function and bound it using Weyl's law:
    \begin{align*}
        F(x) = \left|\left\{ m\in \mathbb{N} \;:\; \lambda_m \leq x \right\}\right| = C_d V_M x^{\frac{d}{2}} + O(x^{ \frac{d-1}{2}}) \lesssim x^{\frac{d}{2}}.
    \end{align*}
	Then the series can be bounded as
    \begin{align*}
        &\sum_{m=1}^{\infty} e^{-\lambda_m t}  \lambda_m^A  = \int_0^\infty  e^{-xt} x^A \;dF(x) = - \int_0^\infty F(x) \left( e^{-tx} x^A \right)' dx  = \int_0^\infty F(x) \left( t e^{-tx}x^A - e^{-tx} A x^{A-1} \right) dx, \\
        &I_1  = \int_0^\infty \hspace{-0.3em}F(x) t e^{-tx} x^A dx \lesssim t \int_0^\infty \hspace{-0.3em}e^{-tx} x^{A + \frac{d}{2}} \;dx = \int_0^\infty \hspace{-0.25em}e^{-y} t^{- \left( A + \frac{d}{2}\right)} y^{A+\frac{d}{2}} \;dy = t^{-\left( A + \frac{d}{2} \right)} \Gamma\left( A+\frac{d}{2}-1 \right), \\
        &I_2 = \int_0^\infty \hspace{-0.3em}F(x) e^{-tx} x^{A-1} \;dx \lesssim t^{-\left( A+\frac{d}{2} \right)} \int_0^\infty e^{-y} y^{A-1+\frac{d}{2}} \leq t^{-\left( A + \frac{d}{2} \right)} \Gamma\left(  A+ \frac{d}{2} \right).
    \end{align*}
	We can put this in \eqref{aux3} to get
	\begin{align*}
        \mathbb{E}W_2\left( \sum_{n=1}^N \frac{1}{N} \delta_{x_n}, \Vol \right) & 
        \lesssim t^{\frac{1}{2}} + \left( N^{a-2} t^{-\left( b-1+\frac{d}{2} \right)} \right)^\frac{1}{2} 
        \underset{t=N^{-\alpha}}{=}  (N^{-\alpha})^\frac{1}{2} + \left( N^{a-2+ \alpha \left( b-1+\frac{d}{2} \right)} \right)^\frac{1}{2}
    \end{align*}
	and, optimizing the exponent in terms of $\alpha$ we get the exponent $\gamma = \frac{2-a}{2b+d}$.
	
	For the case when $b+ \frac{d}{2} = 1$ we need a little more care, as the Gamma integrals do not converge. The first observation is that, because $F(x)=0$ for $0 \leq x \leq \lambda_1$, we can eliminate this part of the integrals and bound them with 
    \begin{align*}
        \sum_{m=1}^\infty e^{-t \lambda_m} \lambda_m^{b-1} &  \lesssim \int_{\lambda_1}^{\infty}F(x) e^{-tx} \left( t x^{b-1} + x^{b-2} \right) dx \lesssim \int_{\lambda_1}^{\infty} x^{\frac{d}{2}} e^{-tx} \left( t x^{b-1} + x^{b-2} \right) dx \\
        & = \int_{\lambda_1}^\infty t e^{-tx} + e^{-tx} x^{-1} dx = \int_{\lambda_1}^\infty e^{-y} (1+y^{-1}) dy \leq 1-\log \lambda_1 t \lesssim 1 - \log t.
    \end{align*}
	We can input this in the smoothing inequality, and after substituting $t=N^{-\alpha}$, we get the bound
    \begin{align*}
        \mathbb{E} W_2\left( \sum_{n=1}^N \frac{1}{N} \delta_{x_n}, \Vol \right) \lesssim t^{\frac{1}{2}} + N^{ \frac{a-2}{2}} \left( 1- \log t \right)^{\frac{1}{2}} = N^{\frac{-\alpha}{2}} + N^{ \frac{a-2}{2}} \left( 1+ \alpha \log N \right)^{\frac{1}{2}} \lesssim N^{ \frac{a-2}{2}} \sqrt{\log N}.
    \end{align*}
	
\end{proof}

When one writes the eigenvalues without repetition $0 = \tilde{\lambda}_0 <  \tilde{\lambda}_1 <  \tilde{\lambda}_2 < \cdots$ and $E_\ell$ the vector space of functions with eigenvalue $\tilde{\lambda}_\ell$, the smoothing inequality may be rewritten as
\begin{align*} 
    \mathbb{E} W_2\left( \sum_{n=1}^N \frac{1}{N} \delta_{x_n}, \Vol \right) & \leq \left( d \cdot t + K(\mathbb{M}) t^{\frac{3}{2}} \right)^{\frac{1}{2}} + 2 \left( \frac{1}{N^2}\sum_{\ell=1}^\infty \frac{e^{- \tilde{\lambda}_\ell t}}{\tilde{\lambda}_\ell} \sum_{\varphi_m\in E_\ell} \mathbb{E} \left( \left| \sum_{n=1}^N \varphi_m(x_n)\right|^2 \right)\right)^{\frac{1}{2}}.
\end{align*}
Using this notation we can take advantage of bounding the sum of some eigenfunctions, instead of each one individually. This may be simpler, as in some cases there are summation formulas that can be exploited. 
\begin{lemma}\label{lemma:smoothing2}
	Let $\left\{ x_n \right\}_{n=1}^N$ be from some point process. Assume the estimates:
	\begin{align*}
		\sum_{\varphi_m \in E_\ell} \mathbb{E} \left( \left|\sum_{n=1}^N \varphi_m(x_n)\right|^2 \right) \lesssim N^a \tilde{\lambda}_\ell^b, && \tilde{\lambda}_\ell \approx \ell^2, && b > \frac{1}{2}.
	\end{align*} 
	Then $\displaystyle \mathbb{E}W_2 \left( \frac{1}{N}\sum_{n=1}^N \delta_{x_n}, \operatorname*{Vol} \right) \lesssim N^{-\gamma}$ for $\gamma = \frac{2-a}{2b+1}$.
\end{lemma}
\begin{proof}
	The proof is identical to the one with repeating eigenvalues, taking into account that when computing the series we now consider the mass function 
	\begin{align*}
		F(x) =  \left| \left\{ \tilde{\lambda}_\ell\;:\; \tilde{\lambda}_\ell \leq x \right\} \right| \lesssim x^{\frac{1}{2}}.
	\end{align*}
\end{proof}
\textbf{Remark:} The first new observation is that while previous results were bounding the next quantity directly, it can also be written using the variance,
\begin{equation} \label{eq:variance}
    \mathbb{E} \left( \left|\sum_{n=1}^N \varphi_m (x_n)\right|^2 \right) = \Var \left( \sum_{n=1}^N \varphi_m (x_n) \right) + \left| \mathbb{E}\left( \sum_{n=1}^N \varphi_m(x_n) \right) \right|^2.
\end{equation}
In this form it is often easier to bound. A key example of this is when the first intensity $\rho_1(x)$ is constant, as then the second term vanishes and it only requires a bound in the variance of a smooth linear statistic.

\subsection{Finite Projection kernels}
\begin{proof}[Proof of Theorem \ref{thm:projectionkernels}]
	Due to the previous remark, it is a direct consequence of the following bound to the variance, applied to the smoothing inequality lemma:
	\begin{align*}
		\Var \left( \sum_{n=1}^N g(x_n) \right) & = \frac{1}{2} \iint |g(x)-g(y)|^2 |K_N(x, y)|^2 dxdy \leq 2 \iint |g(x)|^2 |K_N(x, y)|^2 dxdy \\
		& = 2 \int |g(x)|^2 K_N(x, x) dx = 2N \norm{g}^2_{L_2}. 
	\end{align*}
\end{proof}

\begin{proof}[Proof of Corollary \ref{cor:harmonichom}]
	We only need to prove that the harmonic ensemble has constant first intensity under the assumption that the manifold is homogeneous. For this, recall first that if $g: \mathbb{M} \to \mathbb{M}$ is an isometry, then it commutes with the Laplacian (see \cite[Section 4.4]{Canzani}). Then, if $\varphi \in C^\infty(\mathbb{M})$ is an eigenfunction of the Laplacian, the function $\varphi \circ g$ is also an eigenfunction with the same eigenvalue.  
	
	Note that if we denote the vector space $E_\ell = \left\{ f \;:\; \Delta f = - \tilde{\lambda}_\ell f \right\}$ then one can write the harmonic kernel as the sum of the reproducing kernels of $E_1, E_2, \dots, E_L$: 
	\begin{equation*}
		K_L = \sum_{\ell = 0}^{L} \sum_{\varphi_m \in E_\ell} \varphi_m(x) \overline{\varphi_m(y)}.
	\end{equation*} 
	But given an isometry $g$, we have that $\left\{ \varphi_m \circ g  \right\}_{\varphi_m \in E_\ell}$ is another basis of $E_\ell$ and then
	\begin{equation*}
		\sum_{\ell = 0}^{L} \sum_{\varphi_m \in E_\ell} \varphi_m(x) \overline{\varphi_m(y)} = \sum_{\ell = 0}^{L} \sum_{\varphi_m \in E_\ell} \varphi_m(g(x)) \overline{\varphi_m(g(y))}.
	\end{equation*}   

	The assumption of homogeneity on the manifold guarantees that we can always find $g(x) = a$ for $a \in \mathbb{M}$ fixed, so $K(x, x) = K(a, a)$ is constant. Note that the fact $\Vol (\mathbb{M}) = 1$ implies $K(x, x) = N$.  
\end{proof}

\begin{proof}[Proof of Theorem \ref{thm:harmoniccompact}]
	We need to estimate the two parts in \eqref{eq:variance}. Let's start with the variance. As we don't have now that the first intensity is constant, we can use either Hörmander bound or the following consequence of Weyl's law (see \cite{Hormander_1968}): $K_L(x,x) = N + \mathcal{O}(N^{1 - \frac{1}{d}}) \lesssim N$. Therefore 
	\begin{equation*}
		\Var \left( \sum_{n=1}^N \varphi_m(x_m) \right) \leq 2 \int |\varphi_m(x)|^2 K_L(x, x) dx \lesssim N.
	\end{equation*}
	For the other term, 
	\begin{align*}
        \left|\mathbb{E} \left( \sum_{n=1}^N \varphi_m(x_n) \right)\right| & = \left|\int \varphi_m(x) K_L(x,x) \;dx\right| = \left|\int \varphi_m(x) \left( K_L(x,x)- N \right) dx \right|\lesssim \left|\int |\varphi_m(x)| N^{1- \frac{1}{d}} dx \right|,\\
        \left| \mathbb{E}\left( \sum_{n=1}^N \varphi_m(x_n) \right) \right|^2 & \lesssim N^{2- \frac{2}{d}} \int |\varphi_m(x)|^2dx = N^{2- \frac{2}{d}}.
    \end{align*} 
	To sum up, we have gotten that for dimension $d\geq 2$ 
	\begin{align*}
		\mathbb{E} \left( \left|\sum_{n=1}^N \varphi_m(x_n)\right|^2 \right) \lesssim N + N^{2- \frac{2}{d}} \lesssim  N^{2- \frac{2}{d}}.
	\end{align*}
\end{proof}

\subsection{Two point homogeneous manifolds}

The following result was proved in \cite{levi2023linear} for the sphere $\mathbb{S}^d$. Here it is generalized for the projective spaces, as we can take advantage that the kernel can be expressed in terms of Jacobi polynomials when the manifold is two point homogeneous.

\begin{theorem}\label{thm:projectives}
    Let $\mathbb{M}$ be a two point homogeneous manifold with dimension $d\geq 2$ and $f\in H^{\frac{1}{2}}(\mathbb{M})$. Given $N$ points $\left\{ x_n \right\}$ drawn from the harmonic ensemble, the limit
    \begin{align*}
        \lim_{N \to \infty} \frac{1}{N^{1-\frac{1}{d}}} \Var \left( \sum_{n=1}^N f(x_n) \right) = |||f | | |^2_{\frac{1}{2}} 
    \end{align*}   
	exists, and $ |||f | | |_{\frac{1}{2}}$ is a seminorm equivalent to $|||f | | |_{\frac{1}{2}} \approx_d [f]_{\frac{1}{2}}$.
\end{theorem}
\begin{proof}
	As commented before, the case of the sphere is $\mathbb{M} = \mathbb{S}^d$ \cite[Theorem 2.2]{levi2023linear}, so we can focus on $\mathbb{M} = \mathbb{F}\mathbb{P}^n$. 
	In the projective spaces, the harmonic ensemble has an explicit kernel:
	\begin{align*}
        K_L(x, y) = \frac{\left( \frac{d}{2} + \frac{\dim \mathbb{F}}{2} \right)_L}{ \left( \frac{\dim \mathbb{F}}{2} \right)_L} P_L^{\left( \frac{d}{2}, \frac{\dim \mathbb{F}}{2} -1 \right)} \left( \cos (2d(x, y)) \right) = \frac{N}{\left( L+ \frac{d}{2} \atop L\right)} P_L^{\left( \frac{d}{2}, \frac{\dim \mathbb{F}}{2} -1 \right)} \left( \cos (2d(x, y)) \right)
    \end{align*}
    where the number of points satisfies $\frac{N}{L^d} \to C_d$. Using the formula \eqref{eq:var} for the variance we deduce that the limit we want to compute is
    \begin{align*}
        \lim_{N \to \infty} \frac{1 }{N^{1- \frac{1}{d}}} \Var  \left( \sum_{n=1}^N g(x_n) \right) \approx  \lim_{L \to \infty} \frac{1}{L^{d-1}} \int \int |g(x)-g(y)|^2 K_L(x,y)^2 \;dxdy.
    \end{align*}

    We are going to split the integral into 3 parts, depending on a parameter $0<\varepsilon \ll  1$.

	For the first part, we can use the global bound due to Hörmander:
    \begin{align*}
        K_L(x, y) \leq C \frac{L^{d}}{1+L d(x,y)} \leq C \frac{L^{d-1}}{d(x, y)}.
    \end{align*}
    This gives
    \begin{align*}
        \frac{1}{L^{d-1}} \iint_{d(x,y) < \frac{\varepsilon}{2L}} |g(x) - g(y)|^2 K_L (x, y)^2 dxdy &\lesssim L^{d-1} \iint_{d(x,y) < \frac{\varepsilon}{2L}} \frac{|g(x)-g(y)|^2}{d(x, y)^2} dxdy \\
        &\leq \frac{\varepsilon^{d-1}}{2^{d-1}} \iint_{d(x,y) < \frac{\varepsilon}{2L}} \frac{|g(x)-g(y)|^2}{d(x, y)^{d+1}} dxdy \lesssim \varepsilon^{d-1} [g]_{\frac{1}{2}}^2. \addtocounter{equation}{1} \tag{\theequation} \label{eq:1}
    \end{align*}
	For the points far away, we can use the following lemma, whose proof will be done at the end of this argument.
    \begin{adjustwidth}{2em}{2em}
    \begin{lemma} \label{lemma:szego}
        Let $c$ be a fixed positive constant, $\beta \geq - \frac{1}{2}$ and $\alpha \in \mathbb{R}$. Then for $ \pi - \frac{c}{L} \leq \theta \leq \pi$ we have
        \begin{equation*}
            \left| L \left( \sin \frac{\theta}{2} \right)^{2 \alpha + 1} \left( \cos \frac{\theta}{2} \right)^{2 \beta + 1} P_L^{\left( \alpha, \beta \right)} (\cos \theta)^2 \right| \leq C.
        \end{equation*}
    \end{lemma}
    \end{adjustwidth}
    
    In order to apply it we have to substitute $\alpha = \frac{d}{2}$ and $\beta = \frac{\dim \mathbb{F}}{2} - 1$. Moreover, we will take the convention $\theta = 2d(x, y)$ for the rest of the proof. This way we have
    
    \begin{align*}
        K_L(x, y) = \frac{N}{\left( L + \frac{d}{2}  \atop L \right)} P_L^{\left( \frac{d}{2}, \frac{\dim F}{2} -1 \right)} ( \cos \theta).
    \end{align*}
	It is easy to check that $\displaystyle \frac{N^2}{\left( L + \frac{d}{2}  \atop L \right)^2} \approx L^d$. With these facts we can deduce that
    \begin{align*}
        \frac{1}{L^{d-1}} &\iint_{2d(x,y) > \pi - \frac{\varepsilon}{L}} |g(x) -g(y)|^2 K_L(x,y) dxdy \lesssim \frac{1}{L^{d-1}} \iint_{2d(x,y) > \pi - \frac{\varepsilon}{L}} |g(x)|^2 K_L(x,y) dxdy \\
        &\approx L \int |g(x)|^2 \int_{2d(x,y) > \pi - \frac{\varepsilon}{L}} P_L^{\left( \frac{d}{2}, \frac{\dim F}{2} -1 \right)} (\cos \theta)^2 dy dx \\
        &= L \int |g(x)|^2 \int_{2t > \pi - \frac{\varepsilon}{L}} P_L^{\left( \frac{d}{2}, \frac{\dim F}{2} -1 \right)} (\cos 2t)^2 \frac{1}{\gamma_\mathbb{M}} \sin^{d-1}(t) \cos^{\dim \mathbb{F} -1}(t) \;dt \;dx \\
        &\approx \norm{g}^2_{L^2} \int_{\theta = \pi - \frac{\varepsilon}{L}}^\pi L P_L^{\left( \frac{d}{2}, \frac{\dim F}{2} -1 \right)} (\cos \theta)^2 \sin^{d-1} \left( \frac{\theta}{2} \right) \cos^{\dim \mathbb{F} - 1} \left( \frac{\theta}{2} \right) \;dt \\
        &\lesssim \norm{g}^2_{L^2} \int_{\theta = \pi - \frac{\varepsilon}{L}}^\pi \frac{1}{\sin^2 \left( \frac{\theta}{2} \right)} \approx \norm{g}^2_{L^2} \frac{\varepsilon}{L} \xrightarrow{L \to \infty} 0. \addtocounter{equation}{1} \tag{\theequation} \label{eq:2}
    \end{align*} 
	We can now focus on the middle region
    \begin{equation*}
        \Omega_{\varepsilon, L} = \left\{ (x, y) \in \mathbb{FP}^n \;:\; \frac{\varepsilon}{L} \leq 2d(x, y) = \theta \leq \pi - \frac{\varepsilon}{L} \right\},
    \end{equation*}
    where we will use the following result.
    \begin{adjustwidth}{2em}{2em}
        \begin{lemma*}\cite[Theorem 8.21.13]{Szego_1959}
            Let $\alpha, \beta > -1$. Then for $\frac{c}{L} \leq \theta \leq \pi - \frac{c}{L}$ we have
            \begin{align*}
                P_L^{(\alpha, \beta)} (\cos \theta) &= L^{-\frac{1}{2}} k(\theta) h_L(\theta), \\
                k(\theta) &= \pi^{-\frac{1}{2}}  \left( \sin \frac{\theta}{2} \right)^{- \alpha - \frac{1}{2}} \left( \cos \frac{\theta}{2} \right)^{- \beta - \frac{1}{2}}, \\
                h_L(\theta) &= \cos\left( \left( L + \frac{\alpha + \beta +1}{2} \right) \theta - \left(\alpha + \frac{1}{2}\right) \frac{\pi}{2} \right) + \frac{O(1)}{L \sin \theta}.
            \end{align*}
        \end{lemma*}
    \end{adjustwidth}
    
	After applying it to the harmonic kernel, we have that for $(x, y) \in \Omega_{\varepsilon, L}$
	\begin{align}\label{eq:kernelSzego}
		   K_L(x, y)^2 = \frac{N^2}{L \left( L + \frac{d}{2}  \atop L \right)^2} k^2(\theta) h_L^2(\theta), &&
		k^2(\theta)  = \pi^{-1} \left( \sin \frac{\theta}{2} \right)^{-(d+1)} \left( \cos \frac{\theta}{2} \right)^{-(\dim \mathbb{F} - 1)}.
	\end{align}
    
	We can define now the seminorm
	\begin{equation*}
        \iint |g(x)-g(y)|^2 k\left( 2d(x, y) \right)^2 dxdy =: 2 | | |g| | |^2_{\frac{1}{2}}.
    \end{equation*}
	Because of $k(2d(x,y)) \geq \sin(d(x, y))^{-(d+1)} \geq d(x, y)^{-(d+1)}$ we have $| | |g| | |_{\frac{1}{2}}^2 \gtrsim [g]^2_{\frac{1}{2}}$. For the other inequality, we want to use Proposition \ref{prop:sobolev}, so we have to check that 
	\begin{align*}
        \int_{r \leq d(x, y) \leq 2r} k(2d(x, y))^2 dxdy & \leq \frac{1}{\pi}\int_{r \leq d(x,y)  \leq \frac{\pi}{2}} \frac{dy}{\sin (d(x,y))^{d+1} \cos (d(x,y))^{\dim \mathbb{F} - 1}} \\
        &\approx \int_r^{\frac{\pi}{2}} \frac{(\sin t)^{d-1} (\cos t)^{\dim \mathbb{F} - 1}}{(\sin t)^{d+1} (\cos t)^{\dim \mathbb{F} - 1}} dt = \int_r^{\frac{\pi}{2}} \frac{1}{(\sin t)^2} dt \approx \frac{1}{\tan \frac{r}{4}} \lesssim \frac{1}{r}.
    \end{align*}
	Finally, this proposition gives $|||g| | |^2_{\frac{1}{2}} \lesssim [g]^2_{\frac{1}{2}}$, proving the equivalence of the seminorms.

	We now write $h^2_L(\theta) = \frac{1}{2} + \frac{1}{2} a_{L}(\theta) + b_L(\theta) O(1)$ for
	\begin{align*}
		a_L(\theta) = \cos \left( \theta\left( 2L + \frac{d+\dim \mathbb{F}}{2} \right)  - \frac{\pi}{2} (d+1)\right), &&
        b_L(\theta) = \frac{2}{L\sin \theta} + \frac{1}{L^2 \sin^2 \theta}.
	\end{align*}
	Notice that $b_L(\theta) \xrightarrow{L \to \infty} 0$ pointwise. Because we have the bound $b_L(\theta) \lesssim \frac{1}{\varepsilon} + \frac{2}{\varepsilon^2} \leq \frac{3}{\varepsilon^2}$ we can apply the dominated convergence theorem to get
    \begin{equation*}
        \lim_{L \to \infty}  \iint |g(x)-g(y)|^2 k(\theta)^2 b_L(\theta) O(1) dxdy = 0.
    \end{equation*} 

    If we apply the disintegration theorem to the distance $d: \mathbb{FP}^n \times \mathbb{FP}^n \to [0, \frac{\pi}{2}]$ we get a family of Borel probability measures $\left\{ \mu_t \right\}_{t\in [0, \frac{\pi}{2}]}$ with $\operatorname{supp} \mu_t \subseteq d^{-1}(t) = \left\{ d(x,y)=t \right\}$, and the push-forward measure $\nu = d_*(\Vol \otimes \Vol)$ with expression $d \nu(t) = \frac{1}{\gamma_\mathbb{M}} \sin^{d-1}(t) \cos^{\dim \mathbb{F} - 1} (t) dt$. Because of
    \begin{align*}
        \iint |g(x)-g(y)|^2 k(2d(x,y))^2 dxdy = \int_0^{\frac{\pi}{2}} \iint_{d(x,y)=t} |g(x)-g(y)|^2 k(2t)^2 d \mu_t(x, y) \;d \nu(t)
    \end{align*}
    we deduce that the function
    \begin{equation*}
        F(t) = \frac{1}{\gamma_\mathbb{M}} \sin^{d-1}(t) \cos^{\dim \mathbb{F} - 1} (t) \iint_{d(x,y)=t} |g(x)-g(y)|^2 k(2t)^2 \;d \mu_t(x,y)
    \end{equation*}
    belongs to $F(t) \in L^1([0, \frac{\pi}{2}], dt)$ and, after noticing that 
    \begin{equation*}
        a_L(\theta) \in \left\{ \pm \cos  \left( \theta\left( 2L + \frac{d+\dim \mathbb{F}}{2} \right) \right), \pm \sin  \left( \theta\left( 2L + \frac{d+\dim \mathbb{F}}{2} \right) \right) \right\},
    \end{equation*} 
    we can use the Riemann--Lebesgue lemma to get
    \begin{equation*}
        \iint |g(x)-g(y)|^2k(2d(x, y))^2 a_L(2d(x,y)) \;dxdy = \int_{-\infty}^\infty \carac_{[0, \frac{\pi}{2}]} (t) F(t) a_L(2t) \;dt \xrightarrow{L \to \infty} 0.
    \end{equation*}
\end{proof}
\begin{proof}[Proof of Lemma \ref{lemma:szego}]
    The lemma is a combination of several properties of the Jacobi polynomials. Let's start with \cite[Theorem 8.21.12]{Szego_1959}, which says that for $\alpha > -1$ and an arbitrary real number $\beta$, the following asymptotic is satisfied:
    \begin{align*}
        \left( \sin \frac{\theta}{2} \right)^{\alpha} \left( \cos \frac{\theta}{2} \right)^{\beta} P_{L}^{(\alpha, \beta)} (\cos \theta) = \ell^{- \alpha} \frac{\Gamma (L+ \alpha + 1)}{L!} \left( \frac{\theta}{\sin \theta} \right)^{\frac{1}{2}} J_{\alpha} \left( \ell\theta \right) + \theta^{\alpha+2} O(L^{\alpha})
    \end{align*}
    for $0 < \theta \leq \frac{c}{L}$ and $\ell =  L + \frac{\alpha + \beta + 1}{2} \approx L$. To transform this into a bound, first note that because $L \theta < c$ we have that the error term $\theta^{\alpha+2} O(L^{\alpha}) = O(\theta^2)$ and therefore goes to zero.

    Assuming $\alpha \geq - \frac{1}{2}$ the Bessel function can be bounded by
    \begin{align*}
        \left| J_{\alpha} \left(\ell\theta \right) \right| \lesssim   \left( \ell\theta \right)^{-\frac{1}{2}} \approx (L\theta)^{-\frac{1}{2}}.
    \end{align*}
    The previous inequality is a standard property for $\alpha > - \frac{1}{2}$. The extreme case can be deduced by the explicit formula $J_{-\frac{1}{2}} (t) = \sqrt{ \frac{2}{\pi s}} \cos (s)$.
    
    For the Gamma function, it is well known that 
    \begin{align*}
        \lim_{x \to \infty} \frac{\Gamma(x+ z_0)}{\Gamma(x) z_0^\alpha} = 1, \quad z_0\in \mathbb{C}.
    \end{align*}
    Finally, to include the case $\theta = 0$, one can pass $(\sin \theta)^{\frac{1}{2}} = \left( 2 \sin \frac{\theta}{2} \cos \frac{\theta}{2} \right)^{\frac{1}{2}}$ to the other side. All of this gives
    \begin{align}\label{eq:auxJacobi}
        \left| \left( \sin \frac{\theta}{2} \right)^{\alpha + \frac{1}{2}} \left( \cos \frac{\theta}{2} \right)^{\beta + \frac{1}{2}} P_{L}^{(\alpha, \beta)} (\cos \theta) \right| \lesssim L^{-\frac{1}{2}}.
    \end{align}
    We are interested not in the case $0 \leq \theta < \frac{c}{L}$, but rather the  side $ \pi - \frac{c}{L} < \theta \leq \pi$. Let's denote by $\theta' = \pi - \theta$, so that we have $ \pi - \frac{c}{L} < \theta \leq \pi$ as in the statement and $0 \leq \theta' < \frac{c}{L}$. To pass from one to another, we use the equality \cite[Equation 4.1.3]{Szego_1959} 
    \begin{align*}
        P_L^{(\alpha, \beta)} (x) = (-1)^L P_L^{(\beta, \alpha)} (-x)
    \end{align*}
    to get that $P_L^{(\alpha, \beta)} (\cos \theta) = (-1)^L P_L^{(\beta, \alpha)} (\cos \theta')$. Moreover, this change of variables also translates well with the trigonometry, as $\sin \frac{\theta}{2} = \cos \frac{\theta'}{2}$ and $\cos \frac{\theta}{2} = \sin \frac{\theta'}{2}$. To finish the proof it remains to apply inequality \eqref{eq:auxJacobi} to $P_L^{(\beta, \alpha)} (\cos \theta')$ and square both sides.
\end{proof}

From this theorem, we can deduce a Corollary, which is actually what will be applied.
\begin{corollary}
	For a two-point homogeneous manifold $\mathbb{M}$ with dimension $d\geq 2$ and all functions $f\in H^{\frac{1}{2}}(\mathbb{M})$ with $\norm{f}^2_{L^2} \leq N^{\frac{1}{d}} [f]^2_{\frac{1}{2}}$ the harmonic ensemble satisfies
	\begin{equation*}
		\Var \left( \sum_{n=1}^N f(x_n)  \right) \lesssim N^{1- \frac{1}{d}} [f]^2_{\frac{1}{2}}.
	\end{equation*} 
\end{corollary}
\begin{proof}
	This result is a simplification of the previous result, where the limit is computed. We treat the case of the projective spaces following the proof of Theorem \ref{thm:projectives}. The case of the sphere is obtained analogously by adapting the argument of \cite[Theorem 2.2]{levi2023linear}.

	Fix a small value for $\varepsilon$, let's say $\varepsilon=0.01$.
	\begin{itemize}
		\item In the region $d(x,y) \leq \frac{\varepsilon}{2L}$ the inequalities of \eqref{eq:1} work directly.
		\item In the region $\pi \geq 2 d(x, y) \geq \pi - \frac{\varepsilon}{L}$ we use the hypothesis in \eqref{eq:2}.		\item In the middle region, the proof is shorter as we can directly use that $|h_L(\theta)| \leq C$, and hence
		\begin{align*}
			\frac{1}{L^{d-1}} \iint_{\Omega_{\varepsilon, L}} |g(x)-g(y)|^2 K(x, y)^2 dxdy \approx \iint_{\Omega_{\varepsilon, L}} |g(x)-g(y)|^2 k(\theta)^2 h_L(\theta)^2 dxdy \\
			\lesssim  \iint_{\Omega_{\varepsilon, L}} |g(x)-g(y)|^2 k(\theta)^2 dxdy \lesssim [g]^2_{\frac{1}{2}}.
		\end{align*}
	\end{itemize}
\end{proof}
It is clear we can apply this result to the eigenvalues of the Laplacian, so using this bound in the smoothing inequality lemma proves Theorem \ref{thm:twohomogeneus}.

The main idea of the argument has been to use the Szegö result \eqref{eq:kernelSzego} to write (in most of the manifold) the kernel as 
\begin{equation*}
	K_L(x, y)^2 = C_{d, L} \cdot k(d(x, y)) \cdot h_L(x, y),
\end{equation*}
where $C_{d, L}$ is a constant of the appropriate order, the oscillatory term $h_L$ is bounded independently of $L$, and $k$ satisfies
\begin{equation*}
    \int_{\mathbb{M}} F(d(x, y)) dy \approx \int_{0}^{d_\mathbb{M}} F(r) k(r) dr.
\end{equation*}
A natural continuation would be to generalize this approach for general homogeneous manifolds. As some numerical computations suggest that the above expression for the kernel does not already hold on the torus $\mathbb{T}^2$, this idea seems specific to two-point homogeneous manifolds.

\subsection{Spherical ensemble}
For points from the spherical ensemble and $f\in H^2(\mathbb{S})$, \cite[Theorem 2.5]{levi2023linear} says that
\begin{align*}
	\lim_{N \to \infty} \Var \left( \sum_{n=1}^N f(x_n) \right) = \int_{\mathbb{S}^2} |\nabla f|^2.
\end{align*}  
Note that for the eigenfunctions, the right-hand side is just the eigenvalue. This result can be easily modified to be an inequality that we can use to prove Theorem \ref{thm:spherical}, as the point process has constant first intensity and we only need a bound for the variance. For this, the following two results are used. Denote by $\omega_d$ the volume of $\mathbb{S}^{d}$. \begin{lemma}\label{lemma:harm2}\cite[Proposition 2.15]{levi2023linear}
    Let $d \geq 1$ and $\varrho: (0, \pi) \to \mathbb{R}^+$ a non-negative measurable function, and set 
    \begin{align*}
        Q(\varrho) = \frac{\omega_{d-1}}{\omega_d} \int_0^\pi \varrho(r) \sin^{d-1}(r)dr, && K_{d, p} = \int_{\mathbb{S}^{d-1}} |\langle x, e\rangle |^p d\sigma(x), \quad e\in \mathbb{S}^{d-1}.
    \end{align*}
    If $f\in W^{1, p}(\mathbb{S}^d)$, then for any $1 \leq p < \infty$,
    \begin{align*}
       \iint \frac{|f(x)-f(y)|^p}{d(x,y)^p} \varrho(d(x,y)) d\sigma(x)d\sigma(y) \leq K_{d, p} Q(\varrho) \int_{\mathbb{S}^d} |\nabla f(z)|^p d \sigma (z).
     \end{align*}
\end{lemma}
\begin{lemma}\label{lemma:harm1}\cite[Lemma 2.16]{levi2023linear}
    Let $K_N(x, y) = K_N(d(x,y))$ be the kernel of the $N$ points spherical ensemble on $\mathbb{S}^2$. Let, for $p\in \{1, 2\}$,
    \begin{align*}
        C_{p, N} = \left( \frac{1}{\sqrt{N}} \right)^{2-p} \int_{\mathbb{S}^2} K_N(x, y)^2 d(x, y)^p d\sigma(x).
    \end{align*}
    Then, 
    \begin{equation*}
        \lim_{N \to \infty} C_{1, N} = \sqrt{\pi}, \quad \lim_{N \to \infty} C_{2, N} = 4,
    \end{equation*}
    and,
    \begin{align*}
        \varrho_{p, N} (r) = \frac{1}{C_{p, N}} \left( \frac{1}{\sqrt{N}} \right)^{2-p} K_N(r)^2 r^p, \quad N >0,
    \end{align*}
    is a sequence of radial mollifiers on $\mathbb{S}^2$.
\end{lemma}
Applying this next Corollary to Lemma \ref{lemma:smoothing2} proves Theorem \ref{thm:spherical}.
\begin{corollary}
	For $f\in H^2(\mathbb{S}^2)$ and $x_1, x_2, \dots, x_N \in \mathbb{S}^2$ we have 
	\begin{align*}
		\Var \left( \sum_{n=1}^N f(x_n) \right) \lesssim \int_{\mathbb{S}^2} |\nabla f|^2.
	\end{align*}
\end{corollary}
\begin{proof}
	It requires Lemma \ref{lemma:harm1} to express the variance as an integral with a radial mollifier and then use Lemma \ref{lemma:harm2} to transform it into an inequality, noticing that the constants that depend on the mollifier 
    \begin{equation*}
        \frac{\omega_1}{\omega_2} \int_0^\pi K_N(r)^2 r^2 \sin(r) = C \int_{\mathbb{S}^2} K_N(d(x,y))^2 d(x, y)^2 d\sigma(x)
    \end{equation*}
    are proportional.
\end{proof}

\section{Variations of the Harmonic ensemble on the torus}\label{section:torus}
In the torus $\mathbb{T}^d = \mathbb{R}^d / \mathbb{Z}^d$, the eigenfunctions of the Laplacian are well studied. Indeed, they are $f(x) = e^{2\pi i \langle j, x \rangle }$ for the grid $j\in \mathbb{Z}^d$. Each has as eigenvalue $4 \pi^2 \norm{j}^2_{2}$. For the $p$-norms, $1 \leq p \leq \infty$, one may consider the DPP with kernel 
\begin{align*}
	K_{L, p}(x, y) = \sum_{\norm{j}_p \leq L} e^{2\pi i \langle j, x-y \rangle}, && x, y \in \mathbb{T}^d.
\end{align*}
Note that what we have called the harmonic ensemble is the case $p=2$, while \cite{borda2023riesz} called the case $p=\infty$ the harmonic ensemble. 

Assume $\left\{ x_n \right\}_{n=1}^{N}$ are given by the ``$p$-Harmonic ensemble''. The number of points $N$ is the number of lattice integers in the $p$-ball of radius $L$. Then, because the DPP has a projection kernel with $K(x, x)$ constant, we can use Theorem \ref{thm:projectionkernels} to deduce that
\begin{itemize}
	\item $\displaystyle \mathbb{E} W_2 \left( \frac{1}{N} \sum_{n=1}^N \delta_{x_n}, \operatorname{Vol} \right) \lesssim \frac{1}{N^{\frac{1}{d}}}$ for dimension $d\geq 3$. 
	\item $\displaystyle \mathbb{E} W_2 \left( \frac{1}{N} \sum_{n=1}^N \delta_{x_n}, \operatorname{Vol} \right) \lesssim \frac{\sqrt{\log N}}{N^{\frac{1}{2}}}$ for dimension $d=2$. 
\end{itemize}
This already shows an optimal rate of convergence for high dimensions, so our focus will now be on the case $d=2$. The smoothing inequality lemma and the results for two point homogeneous manifolds suggest that we should aim to prove an inequality
\begin{align*}
	\mathbb{E} \left( \left| \sum_{n=1}^N e^{2\pi i \langle k, x_n \rangle}\right|^2 \right) \lesssim N^{\frac{1}{2}} \norm{k}_2, && k\in \mathbb{Z}^2 \backslash \left\{ 0 \right\}.
\end{align*}
Indeed, this approach is what \cite{borda2023riesz} used for the case $p=\infty$, where they could use the geometry of the $L^\infty$ balls to precisely compute the above value. This can be modified to approximately work in the case $p=1$, but computing the exact value for a general $p$ looks too complicated. 

Denote the $p$-balls by $B_p(x, r) = \left\{ y\in \mathbb{R}^d \;:\; \norm{x-y}_p < r \right\}$.

Following the proof of \cite[Lemma 8]{borda2023riesz}, for $k\in \mathbb{Z}^2 \backslash \left\{ 0 \right\}$, we can express this expected value as follows:
\begin{align*}
	\mathbb{E}_p := \mathbb{E} \left( \left| \sum_{n=1}^N e^{2\pi i \langle k, x_n \rangle}\right|^2 \right) & = \int K(x,x) \left| e^{2\pi i \langle k, x \rangle}\right|^2 dx + \iint e^{2\pi i \langle k,x-y \rangle}\left( K(x,x)K(y,y) - K(x, y)^2 \right)dxdy \\
	& = N - \iint e^{2\pi i \langle k,x-y \rangle} K(x, y)^2 dxdy = N - \hspace{-1.5em}\sum_{ \substack{ l, m \in \mathbb{Z}^d \\ \norm{l}_p, \norm{m}_p \leq L  }}\iint e^{2\pi i \langle k+l-m,x-y \rangle}  dxdy \\
	&= N - \hspace{-1.5em}\sum_{ \substack{ l, m \in \mathbb{Z}^d \\ \norm{l}_p, \norm{m}_p \leq L \\ m-l = k }} 1 \\
	& = N - \# \left\{ \text{lattice points in    } \overline{B}_p \left( 0, L \right) \cap \overline{B}_p (k, L) \right\} \\
	& = \# \left\{ \text{lattice points in    } \overline{B}_p \left( 0, L \right) \backslash \overline{B}_p (k, L) \right\}.
\end{align*}

To prove the desired inequality, we can assume without loss of generality that $0 < \norm{k}_p < \frac{L}{2}$. Because of the inclusion $ \overline{B}_p \left( 0, L \right) \backslash \overline{B}_p (k, L) \subseteq  \overline{B}_p \left( 0, L \right) \backslash \overline{B}_p (0, L- \norm{k}_p)$ we get the inequality
\begin{align*}
	\mathbb{E}_p \leq \# \left\{ \text{lattice points in    } \overline{B}_p \left( 0, L \right) \backslash \overline{B}_p (0, L- \norm{k}_p) \right\}.
\end{align*} 

This latter inequality is more adequate to be estimated in a general setting, as it requires estimates of the lattice points inside the $p$-balls, but not on how they intersect. For instance, the case $p=2$ is already trivial to prove if we use some results of the `Gauss circle problem'. Call $F(r)$ the number of integer points of $\mathbb{Z}^2$ inside the closed ball of radius $r$. Then it is a classical result that
\begin{align*}
	F(r) = \pi r^2 + E(r), && |E(r)| \leq 2 \sqrt{2} \pi r.
\end{align*}   
Which, for our case, implies
\begin{align*}
	\mathbb{E}_2 & \leq F(L) - F(L- \norm{k}_2) = \pi L^2 + E(L) - \pi^2 (L-\norm{k}_2)^2 - E(L-\norm{k}_2) \\ &\leq 2\pi  L \norm{k}_2 + 2 \sqrt{2} \pi (L+L-k) \lesssim L \norm{k}_2 \approx N^{\frac{1}{2}} \norm{k}_2.
\end{align*}

We could look for similar estimations for other norms and repeat the argument, but for the sake of completeness we present an elementary one. Again, we will study the $\norm{\cdot}_2$ norm case first and then comment on its generalizations.

Consider that there are $X$ integer lattice points in the set $\overline{B}_2(R, 0) \backslash \overline{B}_2(r, 0)$ with $1 \leq \frac{R}{2} < r < R$. Then we could put a ball of radius $\frac{1}{2}$ centered at each of the points, so that they are disjoint. Moreover, they would be contained in $\overline{B}_2(R + \frac{1}{2}, 0) \backslash \overline{B}_2(r - \frac{1}{2}, 0)$. Comparing the areas we get
\begin{align*}
	X \frac{\pi}{4} \leq \operatorname*{Area} \left( \overline{B}_2(R + \frac{1}{2}, 0) \backslash \overline{B}_2(r - \frac{1}{2}, 0)  \right) \leq 2 \pi (R+\frac{1}{2}) \cdot (R - r + 1) && \Rightarrow && X \lesssim R \cdot (R- r).
\end{align*}

In our case, we had $R = L$ and $R-r = \norm{k}_2$, so we still get $\mathbb{E}_2 \lesssim N^{\frac{1}{2}} \norm{k}_2$. Note that this argument can be generalized to any other norm, as they are all equivalent, at the cost of adjusting the constants. This is summarized in the following theorem.

\begin{theorem}
	Given a norm in $\mathbb{T}^2$ and $x_1, x_2, \ldots, x_N $ drawn from the DPP with kernel
	\begin{align*}
		K(x, y) = \sum_{\norm{k} \leq L} e^{2\pi i \langle k, x-y \rangle},
	\end{align*}
	we have $\displaystyle  \mathbb{E} W_2 \left( \frac{1}{N} \sum_{n=1}^N \delta_{x_n}, \operatorname{Vol} \right) \lesssim \frac{1}{N^{\frac{1}{2}}}$.
\end{theorem}

\begin{proof}
	In order to apply the smoothing procedure, we can follow the previous argument up until the bound
	\begin{align*}
		\mathbb{E} \left( \left|\sum_{n=1}^N e^{2\pi i \langle k, x_n \rangle}\right|^2 \right) \leq \# \left\{ \text{lattice points in    } \overline{B}_{\norm{\cdot}}(0, L) \backslash \overline{B}_{\norm{\cdot}}(0, L - \norm{k}) \right\} = \mathbf{X}, && k\in \mathbb{Z}^2,\;\; 0< \norm{k} < \frac{L}{2}.
	\end{align*}
	We have used the natural notation 
	\begin{align*}
		B_{\norm{\cdot}} (x, r) = \left\{ y\in \mathbb{R}^2 \;:\; \norm{x-y} < r \right\}.
	\end{align*}
	Because all norms are equivalent, there exist constants $0 < a \leq 1 \leq A < \infty$ such that $a \norm{\cdot}_2  \leq \norm{\cdot} \leq A \norm{\cdot}_2$. This translates to the following relationship with the Euclidean balls:
	\begin{align*}
		B_2 (x, \frac{r}{A}) \subseteq B_{\norm{\cdot}} (x, r) \subseteq B_2(x, \frac{R}{a}).
	\end{align*}  
	Then again placing a ball of radius $\frac{1}{2}$ (in the $\norm{\cdot}$ topology) around each point of the lattice and passing to the Euclidean topology when needed, we get the chain of inequalities
	\begin{align*}
		\mathbf{X} \cdot \operatorname{Area} \left( B_2 (0, \frac{1}{2A}) \right) & \leq  \operatorname{Area} \left(  \overline{B}_{\norm{\cdot}}(0, L+ \frac{1}{2}) \backslash \overline{B}_{\norm{\cdot}}(0, L - \norm{k} - \frac{1}{2}) \right) \lesssim \left( L+\frac{1}{2} \right) \left( \norm{k} + 1 \right) \lesssim L \norm{k}.
	\end{align*} 
	This finishes the proof because $N \approx L^2$ due to the equivalence of the norms. 
\end{proof}

\begin{corollary}
	Let $\mathbb{T}^2 = \mathbb{R}^2 / \Gamma$ be a general torus with lattice $\Gamma$. For $\left\{ x_n \right\}_{n=1}^N$ given by the harmonic ensemble we have
	\begin{align*}
		\mathbb{E} W_2 \left( \frac{1}{N} \sum_{n=1}^N \delta_{x_n}, \operatorname{Vol} \right) \lesssim \frac{1}{N^{\frac{1}{2}}}.
	\end{align*} 
\end{corollary}
\begin{proof}
	Now the eigenfunctions take the parameters not on the grid $\mathbb{Z}^2$ but on the dual lattice $\Gamma^*$, but we can repeat the argument of the regular torus up to the point 
	\begin{align*}
		\mathbb{E} \left( \left| \sum_{n=1}^N e^{2\pi i \langle k, x_n \rangle}\right| \right) = \# \left\{ \text{ lattice points of $\Gamma^*$ in    } \overline{B}(0, L) \backslash \overline{B}(k, L) \right\}, && k\in \Gamma^*\backslash \left\{ 0 \right\}.
	\end{align*}
	We can reduce this to the regular torus case by noticing that this quantity is the number of points in the lattice $\mathbb{Z}^2$ for the balls of some norm $\norm{\cdot}$. If the dual lattice $\Gamma^*$ is generated by the vectors $v$ and $w$ this norm is $\norm{(x, y)} := \norm{xv+ yw}_{L^2}$.
\end{proof}

\section{Zero set of the spherical GAF}

Sodin and Tsirelson showed the variance has the following asymptotic for the zeros $z_1, z_2, \dots, z_N \in \mathbb{C}$ of the spherical GAF
\begin{align*}
    \Var \left( \sum_{z_n=1}^N \varphi(z) \right) = \frac{\kappa}{N} \norm{\Delta_{\mathbb{S}^2} \varphi}^2_{L^2(m^*)} + o\left( \frac{1}{N} \right)
\end{align*} 
for $\varphi\in C^\infty_c(\mathbb{C})$. We are interested in getting a similar result but with a uniform bound $o\left( \frac{1}{N} \right)$ for the family of eigenfunctions of the Laplacian on the sphere.
\begin{theorem}
	For $Y_\ell^m : \mathbb{S}^2 \to \mathbb{R}$ the $m_\ell = 2\ell +1$ eigenfunctions with $\tilde{\lambda}_\ell = \ell(\ell +1)$ we have
    \begin{align*}
        \sum_{m} \Var \left( \sum_{x_n=1}^N Y_m^\ell(x_n) \right) \leq \frac{m_\ell \tilde{\lambda}_\ell^2}{N} \cdot \frac{\pi^2}{6}.
    \end{align*}   
\end{theorem}

Because the point process has constant first intensity (on the sphere), when applying this result to the smoothing inequality (Lemma \ref{lemma:smoothing2}) we get the optimal asymptotics of the Wasserstein distance, proving Theorem \ref{thm:zerosGAF}.

\begin{proof}
	First note that we can compute the variance in the complex plane instead of on the sphere. We will use the notation $Y_m^\ell(z)=: Y_m^\ell(F(z))$ for the standard stereographic projection $F: \mathbb{C} \to \mathbb{S}^2$. When we push forward the background of the sphere and its Laplacian they have the expression
	\begin{align*}
		dm^*(z) = \frac{1}{\pi(1+|z|^2)^2} dm(z), && \Delta_{\mathbb{S}^2} = \frac{1}{4}(1+|z|^2)^2 \Delta_{\mathbb{C}}.
	\end{align*}    
	We can use that for an analytic function $f:\mathbb{C} \to \mathbb{C}$ and $\varphi\in C^\infty_c(\mathbb{C})$ we have the formula
	\begin{equation*}
        \sum_{z\in f^{-1}(0)} \varphi(z) = \int_\mathbb{C} \Delta \varphi(z) \frac{1}{2\pi} \log |f(z)| dm(z).
    \end{equation*}
	This allows us to express the sum of the zeros of our GAF as an integral, which actually is quite similar to its expected value due to the Edelman--Kostlan formula for the first intensity. Indeed, this is what it is used in \cite[Section 3.5]{Hough_Krishnapur_Peres_2012} to prove the asymptotics for one function. We can follow this proof until we arrive at formula (3.5.4):
	\begin{align} \label{aux:4}
        \mathbb{E} \left( \sum_{z_n=1}^N \varphi(z_n) \right) = 4\int_{\mathbb{C}^2} \Delta_{\mathbb{S}^2} \varphi(z) \Delta_{\mathbb{S}^2} \varphi(w) \rho_N(z, w) dm^*(z) dm^*(z)
    \end{align}
	where $\displaystyle \rho_N (z, w) := \operatorname*{Cov} \left( \log \frac{|f_N(z)|}{\sqrt{K(z, z)}}, \log \frac{|f_N(w)|}{\sqrt{K(w, w)}} \right)$ can be expressed as
	\begin{align*}
		\theta(z, w) = \frac{1+ z\overline{w}}{ \sqrt{(1+|z|^2) (1+|w|^2)}}, && \rho_N(z, w) = \sum_{m=1}^\infty \frac{|\theta(z, w)|^{2Nm}}{4m^2} \geq 0.
	\end{align*}
	We are going to use the series to bound the following integral, being a key element that $\rho_N$ is invariant under rotations:
	\begin{align*}
		\int_{\mathbb{C}^2} \rho_N(z, w) dm^*(z) dm^*(w) &= \int_{\mathbb{C}^2} \rho_N(z, 0) dm^*(z) dm^*(w) = \sum_{m=1}^\infty  \frac{1}{4m^2} \int_\mathbb{C} \frac{1}{\pi (1+|z|^2)^{Nm+2}} dm(z) \\
		&= \sum_{m=1}^\infty \frac{1}{4m^2} \int_{r=0}^\infty \frac{2r}{(1+ r^2)^{Nm+2}} dr = \sum_{m=1}^\infty \frac{1}{4m^2} \int_{t=0}^\infty \frac{1}{(1+t)^{Nm+2}} \\
		&= \sum_{m=1}^\infty \frac{1}{4m^2} \frac{1}{Nm+1} \leq \frac{1}{N} \frac{\pi^2}{4\cdot6}.
	\end{align*}
	This suggests we can try to use dominated convergence in \eqref{aux:4}. For a fixed eigenvalue $Y_m^\ell= \varphi$ we need to find an adequate sequence to approximate it. Choose a family of test functions $\xi_k:\in C^\infty_c(\mathbb{C})$ with
    \begin{align*}
        0 \leq \xi_k \leq 1, && |\nabla \xi_k| \leq C, && |\Delta \xi_k| \leq C, && \xi_k |_{\overline{B}(0, k)} \equiv 1, && \xi_k |_{\overline{B}(0, k+1)} \equiv 0.
    \end{align*} 
    Then $\varphi_k := \xi_k \varphi$ are smooth with compact support, converge pointwise to $\varphi$, the Laplacian is
    \begin{equation*}
        \Delta \varphi_k = \Delta \xi_k \varphi + \xi_k \Delta \varphi + \nabla \xi_k \cdot \nabla \varphi_k
    \end{equation*}
    and therefore converges pointwise to $\Delta \varphi$ and is bounded. This means that
    \begin{align*}
        \Var\left( \sum_{n=1}^N Y_m^\ell(z_n) \right) & = 4 \int_{\mathbb{C}^2} \Delta_{\mathbb{S}^2} Y^\ell_m(z) \Delta_{\mathbb{S}^2} Y^\ell_m(w) \rho_L(z, w) dm^*(z) dm^*(w) 
        \\ &= 4\tilde{\lambda}_l^2 \int_{\mathbb{C}^2} Y^\ell_m(z)  Y^\ell_m(w) \rho_L(z, w) dm^*(z) dm^*(w), \\
        \sum_{m}  \Var\left( \sum_{n=1}^N Y_m^\ell(z_n) \right) & = 4\tilde{\lambda}_\ell^2 \int_{\mathbb{C}^2} F(z, w) \rho_L(z, w) dm^*(z) dm^*(w),
    \end{align*}
	where $F: \mathbb{C} \times \mathbb{C} \to \mathbb{R}$ is the reproducing kernel of the space $\left\{\Delta_{\mathbb{S}^2} f = \tilde{\lambda}_l f \right\}$ (with the measure $dm^*$). Note that because this kernel is invariant under rotations we can bound it by the dimension of the space. Combining this with the computation of the integral of $\rho_N$ finishes the proof.
\end{proof}

\printbibliography

@article{levi2023linear,
    author = {Levi, Matteo and Marzo, Jordi and Ortega-Cerdà, Joaquim},
    title = {Linear {S}tatistics of {D}eterminantal {P}oint {P}rocesses and {N}orm {R}epresentations},
    journal = {International Mathematics Research Notices},
    volume = {2024},
    number = {19},
    pages = {12869-12903},
    year = {2024},
    doi = {10.1093/imrn/rnae182},
}

@article{borda2023riesz,
author = {Borda, Bence and Grabner, Peter and Matzke, Ryan W.},
title = {Riesz energy, discrepancy, and optimal transport of determinantal point processes on the sphere and the flat torus},
journal = {Mathematika},
volume = {70},
number = {2},
pages = {e12245},
doi = {https://doi.org/10.1112/mtk.12245},
year = {2024}
}

@article{borda2022empirical,
   title={Empirical measures and random walks on compact spaces in the quadratic {W}asserstein metric},
   volume={59},
   ISSN={0246-0203},
   DOI={10.1214/22-aihp1322},
   number={4},
   journal={Annales de l’Institut Henri Poincaré, Probabilités et Statistiques},
   publisher={Institute of Mathematical Statistics},
   author={Borda, Bence},
   year={2023} }

@book {Szego_1959,
    AUTHOR = {Szeg\H{o}, G\'{a}bor},
     TITLE = {Orthogonal polynomials},
    SERIES = {American Mathematical Society Colloquium Publications, Vol.
              XXIII},
   EDITION = {Fourth},
 PUBLISHER = {American Mathematical Society, Providence, R.I.},
      YEAR = {1975},
   MRCLASS = {42A52 (33A65)},
  MRNUMBER = {0372517},
}

@article{Brandolini_Choirat_Colzani_Gigante_Seri_Travaglini_2014, title={Quadrature rules and distribution of points on manifolds}, DOI={10.2422/2036-2145.201103_007}, journal={ANNALI SCUOLA NORMALE SUPERIORE - CLASSE DI SCIENZE}, author={Brandolini, Luca and Choirat, Christine and Colzani, Leonardo and Gigante, Giacomo and Seri, Raffaello and Travaglini, Giancarlo}, year={2014}, pages={889–923}}

@book{Hough_Krishnapur_Peres_2012, place={Providence (R.I.)}, title={Zeros of {G}aussian analytic functions and {D}eterminantal {P}oint {P}rocesses}, publisher={American Mathematical Society}, author={Hough, John Ben and Krishnapur, Manjunath and Peres, Yuval}, year={2012}}

@book{villani_2009, 
    place={Berlin, Heidelberg}, 
    title={Optimal transport: {O}ld and new}, 
    publisher={Springer Berlin Heidelberg}, 
    author={Villani, Cédric}, 
    year={2009}
}

@article{imbert2019weak, title={The weak {H}arnack inequality for the Boltzmann equation without cut-off}, volume={22}, DOI={https://doi.org/10.4171/jems/928}, number={2}, journal={Journal of the European Mathematical Society}, publisher={European Mathematical Society}, author={Imbert, Cyril and Silvestre, Luís}, year={2019}, pages={507–592} }

@article{Sogge_1987, title={On the convergence of {R}iesz means on compact manifolds}, volume={126}, DOI={10.2307/1971356}, number={3}, journal={The Annals of Mathematics}, author={Sogge, Christopher D.}, year={1987}, pages={439}}

@article{Wang52,
 ISSN = {0003486X},
 author = {Hsien-Chung Wang},
 journal = {Annals of Mathematics},
 number = {1},
 pages = {177--191},
 publisher = {Annals of Mathematics},
 title = {Two-{P}oint {H}omogeneous {S}paces},
 volume = {55},
 year = {1952}
}

@misc{anderson2022riesz,
      title={{R}iesz and {G}reen energy on projective spaces}, 
      author={Austin Anderson and Maria Dostert and Peter J. Grabner and Ryan W. Matzke and Tetiana A. Stepaniuk},
      year={2022},
}

@misc{Canzani, url={https://canzani.web.unc.edu/wp-content/uploads/sites/12623/2016/08/Laplacian.pdf}, title={Analysis on manifolds via the {L}aplacian}, author={Canzani, Yaiza}}

@article{BELTRAN20191073,
title = {A generalization of the spherical ensemble to even-dimensional spheres},
journal = {Journal of Mathematical Analysis and Applications},
volume = {475},
number = {2},
pages = {1073-1092},
year = {2019},
issn = {0022-247X},
author = {Carlos Beltrán and Ujué Etayo}
}

@misc{beltran2017projective,
      title={The projective ensemble and distribution of points in odd-dimensional spheres}, 
      author={Carlos Beltrán and Ujué Etayo},
      year={2017}
}

@article{Hormander_1968, title={The spectral function of an elliptic operator}, volume={121}, DOI={10.1007/bf02391913}, journal={Acta Mathematica}, author={Hörmander, Lars}, year={1968}, pages={193–218}}

@book{Buser_2010, title={Geometry and spectra of compact {R}iemann surfaces}, ISBN={9780817649920}, publisher={Springer Science and Business Media}, author={Buser, Peter}, year={2010} }
\end{document}